\documentclass[12pt]{article}

\usepackage{amsmath, amssymb, amsthm, fullpage, xypic, rotating, lscape, color, comment}

\theoremstyle{plain}
\newtheorem{theorem}{Theorem}

\newtheorem{proposition}[theorem]{Proposition}
\newtheorem{claim}[theorem]{Claim}
\newtheorem{lemma}[theorem]{Lemma}

\newtheorem{rlemma}{Lemma}
\theoremstyle{definition}
\newtheorem{definition}[theorem]{Definition}
\newtheorem{remark}[theorem]{Remark}

\newcommand{\forces}{\Vdash}
\newcommand{\hook}{\upharpoonright}

\newcommand{\tie}{^\frown}

\newcommand{\SCn}{\mathbb{SC}_n}

\title{The Implicitly Constructible Universe}

\includecomment{comment1}
\includecomment{comment2}
\includecomment{comment3}
\author{Marcia J. Groszek\\Dartmouth College\\6188 Kemeny Hall\\Hanover NH 03755-3551\\  \\Joel David Hamkins\\City University of New York\\The CUNY Graduate Center \&\\College of Staten Island of CUNY}

\begin{document}

\maketitle

\begin{abstract}
We answer several questions posed by Hamkins and Leahy concerning the \emph{implicitly constructible universe} $Imp$, which they introduced in \cite{HL}.  Specifically, we show that it is relatively consistent with ZFC that $Imp \models \neg \text{CH}$, that $Imp \neq \text{HOD}$, and that $Imp \models V \neq Imp$, or in other words, that $(Imp)^{Imp} \neq Imp$.
\end{abstract}

\section{Introduction}

The \emph{implicitly constructible universe}, denoted $Imp$, was defined by Hamkins and Leahy \cite{HL}:

\begin{definition}
For a transitive set $X$, a subset $S \subseteq X$ is \emph{implicitly definable} over $X$ if for some formula $\varphi(x_1, \dots, x_n)$ in the language of $ZFC$ with an additional one-place predicate symbol, and some parameters $a_1, \dots, a_n \in X$, the set $S$ is the unique subset of $X$ such that $$( X, \in, S) \models \varphi (a_1, \dots, a_n).$$
\end{definition}

\begin{definition}  $Imp$ is defined by iteratively applying the implicitly definable power set operation as follows.

$Imp_0 = \emptyset$;

$Imp_{\alpha + 1} = \{S \;|\; S \text{ is implicitly definable over }Imp_\alpha\}$ ($Imp_1 = \{\emptyset\}$);

$Imp_\lambda = \displaystyle\bigcup_{\alpha < \lambda} Imp_\alpha$ for limit $\lambda$.

$Imp = \displaystyle\bigcup_{\alpha \in OR} Imp_\alpha$.

\end{definition}

Hamkins and Leahy showed the following facts.

\begin{proposition}[Hamkins and Leahy \cite{HL}]  \label{prop-HL}

$Imp$ is an inner model of $ZF$, with $L \subseteq Imp \subseteq HOD$.

If $ZF$ is consistent, so is $ZFC + (Imp \neq L)$.

For $\alpha < \omega_1^L$, as a consequence of Shoenfield absoluteness, $Imp_\alpha = (Imp_\alpha)^L$; thus, $Imp_{\omega_1^L} = (Imp_{\omega_1^L})^L = L_{\omega_1^L}$.

\end{proposition}

In this paper, we answer some questions posed by Hamkins and Leahy \cite{HL}.   These questions aim to separate  $Imp$ from $L$ and from $HOD$, both literally  (Hamkins and Leahy showed that we may have $Imp \neq L$, and we show here that we may have $Imp \neq HOD$) and in terms of their properties.  In particular, we show that (given the consistency of $ZF$):
\begin{enumerate}
\item  It is consistent that $Imp \models \neg CH$.  (Theorem~\ref{thm-notCH}.)
\item  It is consistent that $Imp \neq HOD$.  (Theorem~\ref{thm-ImpNotHOD}.)
\item  It is consistent that $Imp \models V \neq Imp$ (that is, $(Imp)^{Imp} \neq Imp$).  (Theorem~\ref{thm-ImpImpNotImp}.)
\end{enumerate}

$Imp$ is defined level-by-level, inductively, as is the constructible universe $L$.  An important distinction is that, given $L_\alpha$ and $S \subset L_\alpha$, whether $S \in L_{\alpha + 1}$ depends only on $L_\alpha$, while given $Imp_\alpha$ and $S \subset Imp_\alpha$, whether $S \in Imp_{\alpha + 1}$ depends on the power set of $Imp_\alpha$.  By the above results, despite the similarity of the definitions of $Imp$ and $L$, we see that $Imp$ is less like $L$ and more like $HOD$.

We leave open, among other things, the question of whether $Imp \models \neg AC$ is consistent.  As $(V = Imp) \Longrightarrow AC$, Theorem~\ref{thm-ImpImpNotImp} is relevant to this question.

\section{Preliminary Results and Notation}
\label{sec-prelim}

The following facts were noted by Hamkins and Leahy:

\begin{proposition}[Hamkins and Leahy \cite{HL}]
\label{prop-unique}  Suppose $M \models ZFC + V=L$.  If $\mathbb P$ is a forcing poset in $M$, and $G$ is the unique element of $M[G]$ that is $\mathbb P$-generic over $M$, then in $M[G]$ we must have $G \in Imp$, and therefore $(Imp)^{M[G]} = M[G]$.
\end{proposition}

\begin{proposition}[Hamkins and Leahy \cite{HL}]
\label{prop-almosthomog}  If $G$ is $\mathbb P$-generic over $M$, where $\mathbb P$ is an almost-homogeneous notion of forcing in $M$, then $M[G] \models Imp \subseteq M$.
\end{proposition}

Hamkins and Leahy use Proposition~\ref{prop-unique} in the proof of the consistency of $Imp \neq L$ \cite{HL}.

By playing off rigidity (unique generics) against (almost) homogeneity, we can control which sets belong to $Imp$ in generic extensions and in their submodels.

We begin with the following proposition:

\begin{theorem}
\label{thm-notCH}
  If $ZF$ is consistent, so is $ZFC + (Imp \models \neg CH)$.
\end{theorem}

\begin{proof}
Abraham's model \cite{abraham} and Groszek's model \cite{groszek}
 for a minimal failure of $CH$ are each produced by forcing over a model $M$ of $V=L$ with a poset $\mathbb P$ such that $G$ is the unique element of $M[G]$ that is  $\mathbb P$-generic over $M$, and $M[G] \models \neg CH$.  Since, by Proposition~\ref{prop-unique}, in each of these models $Imp = M[G]$, it follows that in each of these models $Imp \models \neg CH$.
\end{proof}

To prove the main theorems of this paper, we will employ the technique of Groszek \cite{groszek} to produce unique generics.  This entails coding a generic sequence of reals into the degrees of constructibility of the generic extension, by combining products and iterations of Sacks forcing.

For the remainder of this paper, will will use the following notation.

We identify a set with its characteristic function, so if $x \in 2^\alpha$, we may say $\beta \in x$ rather than $x(\beta) = 1$.

We let $M$ be a model of $ZFC + V=L$, and define forcing partial orders in $M$.

$\mathbb S$ will denote Sacks forcing.  Forcing with $\mathbb S$ over a model of $V=L$ adds a generic real $g \subseteq 2^\omega$ of minimal (nonzero) $L$-degree \cite{sacks}.

$\mathbb Q$ will always denote a countable-support iteration $\left< \mathbb Q_\beta \;|\; \beta < \alpha \right>$
 of some countable length $\alpha$, such that each $\mathbb Q_\beta$ is forced to be either $\mathbb S$ or $\mathbb S \times \mathbb S$, and $\mathbb Q_0 = \mathbb S$.  Whether $\mathbb Q_\beta$ is $\mathbb S$ or $\mathbb S \times \mathbb S$ may depend on the generic sequence below $\beta$.

 The following proposition follows from earlier work on Sacks forcing (for example, see Baumgartner and Laver \cite{BL}, and Groszek \cite{groszek}).

 \begin{proposition}
 \label{prop-degstruc}
  Forcing with $\mathbb Q$ preserves $\omega_1$, and if $G$ is $\mathbb Q$-generic over a model $M$ of $V=L$, then in $M[G]$ the $L$-degrees of reals are exactly:

 \begin{enumerate}
 \item a well-ordered sequence $\left< d_\beta \;|\; \beta \leq \alpha\right>$, where $d_0$ is the degree of $\emptyset$, $d_{\beta + 1}$ is the degree of the $\mathbb Q_\beta$-generic, and for limit $\beta$, $d_\beta$ is the degree of the sequence of generic reals $\left< d_\gamma \;|\; \gamma < \beta\right>$; and
 \item  for each $\beta < \alpha$ such that $\mathbb Q_\beta$ is $\mathbb S  \times \mathbb S$, a pair of incomparable degrees $d_{\beta,0}$ and $d_{\beta,1}$ between $d_\beta$ and $d_{\beta + 1}$.
\end{enumerate}

\end{proposition}

$\mathbb P$ will always denote a countable-support product $\displaystyle \prod_{i \in I} \mathbb P_i$, where each $\mathbb P_i$ has the form $\mathbb Q$ described above.

We will denote the $\mathbb P$-generic sequence by $G = \left< G_i \;|\; i \in I\right>$, where $G_i$ is $\mathbb P_i$-generic.

Each $G_i$ is equivalent to a sequence of generic reals of countable length $\alpha_i$; we will denote the join of these reals (relative to some fixed counting of $\alpha_i$) as $g_i$.

An important technical lemma is the following.

\begin{lemma}
\label{lemma-technical}
Suppose $M \models V=L$, the poset $\mathbb P \in M$ is as described above, and $G$ is $\mathbb P$-generic over $M$, where $G = \left< G_i \;|\; i \in I\right>$.  If $x$ is a real in $ M[G]$ and for all $i \in I$ we have $x \not\in M[G_i]$, then the $L$-degree of $x$ lies above at least two minimal (nonzero) $L$-degrees of reals.
\end{lemma}

This lemma is proven using a fusion construction of the sort common to Sacks forcing arguments.  Since the proof is not especially illuminating of any new ideas, we defer it to Section~\ref{sec-technical}, at the end of the paper.  At the end of Section~\ref{sec-technical} we state a more general result about degrees of constructibility of reals in generic extensions by forcing notions built from Sacks forcing.

\section{Separating $Imp$ from $HOD$}

In this section, we produce a model $N$ in which $Imp \neq HOD$.

\begin{definition}
Let $M$ be a model of $V=L$, and in $M$, let $\mathbb P = \displaystyle\prod_{i \in I} \mathbb P_i$ be the countable-support product defined by letting $I$ be $ \omega_1 \times \omega_1$ and $\mathbb P_{(\alpha, \beta)}$ be the length $\alpha$ iteration of Sacks forcing $\mathbb S$.
\end{definition}

By Proposition~\ref{prop-degstruc}, each $\mathbb P_{(\alpha,\beta)}$ adds an initial segment of degrees of constructibility of reals of order type $\alpha + 1$,  which we will call a tower of height $\alpha + 1$, with top point $deg(g_{(\alpha,\beta)})$.  (By the conventions stated after Proposition~\ref{prop-degstruc}, $g_{(\alpha,\beta)}$ denotes the join of the sequence of reals added by $\mathbb P_{(\alpha,\beta)}$.)  Furthermore, in $M[G]$, the only well-ordered initial segments of the degrees of constructibility of reals are these towers and their initial segments.  (To see this, suppose $x$ is a real whose degree is not in one of these towers,  Then for all $i \in \omega_1 \times \omega_1$, we have $x \not\in M[G_i]$.  Therefore, by the technical lemma (Lemma~\ref{lemma-technical}), $x$ lies above at least two minimal (nonzero) $L$-degrees of reals, so its degree is not in any well-ordered tower of $L$-degrees.)  Hence, each of these towers is maximal.  In our argument later, we will code information into certain submodes of the forcing extension by controlling the ordinals $\alpha$ for which there is a unique such maximal tower.

In $M[G]$ there are also Cohen subsets of $\omega_1$, that is, subsets of $\omega_1$ that are
generic over $M$ for the forcing $Add(\omega_1, 1)$ whose conditions are countable partial functions from $\omega_1$ to $2$.  One such element is $x$, defined by $x(\alpha) = g_{(\alpha, 0)}(0)$.

Define the model $N$ by $N= M[H]$, where
$$H = \left< G_{(\alpha, \beta)}\;|\; (\gamma \in Lim \cup \{0\} \;\&\;  n \in \omega  \;\&\;   \alpha = \gamma + 2n \;\&\; x(\gamma + n ) = 0) \;\Rightarrow\; \beta = 0 \right>.$$
That is, for $\gamma \in Lim \cup \{0\}$, if $x(\gamma + n) = 0$ then we omit from $N$ all but one maximal tower of height $\gamma + 2n + 1$.

\begin{claim}
\label{claim-noextragenerics}  If $G_{(\alpha, \beta)}$ is not an entry in the sequence $H$, then no real whose (nonzero) degree is in the tower added by $G_{(\alpha, \beta)}$ is in $N$.

Therefore, the maximal towers in $N$ are precisely those added by the $G_{(\alpha, \beta)}$, where for some $\gamma \in Lim \cup \{0\}$ and $n \in \omega$, we have either $ \alpha = \gamma + 2n \;\&\; (x(\gamma + n ) = 0 \;\Rightarrow\; \beta = 0)$ or $\alpha = \gamma + 2n + 1$.
\end{claim}

\begin{proof}
Suppose $p$ forces that $G(\alpha,\beta)$ is not an entry in $H$.  That is, for some $\gamma \in Lim \cup \{0\}$ and some $n \in \omega$, we have $\alpha = \gamma + 2n$ and $\beta \neq 0$ and $p$ forces that $x(\gamma + n) = 0$ (that is, $p \forces g_{\alpha, 0}(0) = 0$).  Since $p$ forces that $g_{\alpha, 0}(0) = 0$, then $p$ forces that $H$ is an element of
$$M' = M[\left< G_{\alpha', \beta'} \;|\; \alpha' = \alpha \;\Longrightarrow\; \beta' = 0\right>].$$
But since $\left< G_{\alpha', \beta'} \;|\; \alpha' = \alpha \;\Longrightarrow\; \beta' = 0\right>$ is generic over $M$ for $\displaystyle\prod_{i \in J} \mathbb P_i$, where $J \subset I$ is in $M$ and $(\alpha, \beta) \not\in J$, it follows by standard results about product forcing that no element of $M[G_{(\alpha,\beta)}] \setminus M$ is in $M'$.  Since $M' \supseteq N$, this completes the proof.
\end{proof}

\begin{claim}  In $N$, $x$ is ordinal definable
\end{claim}

\begin{proof}
For $\gamma \in Lim\cup\{0\}$, recalling that $G_{(\alpha,\beta)}$ adds a maximal tower of height $\alpha + 1$, we have $x(\gamma + n) = 0$ iff there is a unique maximal tower of $L$-degrees of reals of height $\gamma + 2n + 1$.  (If $x(\gamma + n) = 1$, there are $\omega_1$-many maximal towers of height $\gamma + 2n + 1$.  There are always $\omega_1$-many maximal towers of height $\gamma + 2n + 2$.)
\end{proof}

Furthermore, in $N$, there is $G'$ that is $\mathbb P$-generic over $M$, defined by, for $\gamma, \rho \in Lim \cup \{0\}$ and $n,m \in \omega$,
$$G'(\gamma + 2n, \rho + m) = G(\gamma + 2n + 1, \rho + 2m))\hook \gamma + 2n;$$
$$G'(\gamma + 2n + 1, \rho + m) = G(\gamma + 2n + 1, \rho + 2m + 1);
$$
where $G(\alpha,\beta) \hook \delta$, for $\delta < \beta$, is $\{ p \hook \delta \;|\; p \in G(\alpha, \beta)\}$.  That is, we are recovering the existence of many mutually generic maximal towers of $L$-degrees of reals of height $\alpha + 1$ for all $\alpha$, by cutting half the towers of height $\gamma + 2n + 2$ down to height $\gamma + 2n + 1$.

Now we have $M[G'] \subset N \subset M[G]$.

\begin{claim}
$(Imp)^N = M$.
\end{claim}

\begin{remark}  The following proof generalizes to show:

Suppose $M \subseteq M' \subseteq N \subseteq M''$ are models of set theory.  Suppose further that $Imp^{M'} \subseteq M$, $Imp^{M''} \subseteq M$, and $M'$ and $M''$ satisfy the same sentences with parameters from $M$.  (This will follow if $M'$ and $M''$ are generic extensions of $M$ via the same almost homogeneous forcing notion in $M$.)  Then
$$Imp^{M'} = Imp^N = Imp^{M''}.$$
\end{remark}

\begin{proof}
The forcing $\mathbb P$ is almost homogeneous, so by Proposition~\ref{prop-almosthomog}, $(Imp)^{M[G']} = (Imp)^{M[G]} = M$.  In fact, for all $\alpha$, we have
$$(Imp_\alpha)^{M[G']  }  = (Imp_\alpha)^{ M[G] } = \{y \in M \;|\; \forces_{\mathbb P} (y \in Imp_\alpha)\},$$
and so we may define $I_\alpha = (Imp_\alpha)^{M[G']  } = (Imp_\alpha)^{M[G]  }  \in M$.

Show inductively that, for all $\alpha$,
$$(Imp_\alpha)^{ N } = I_\alpha.$$

Assume as inductive hypothesis that $(Imp_\alpha)^{ N } = I_\alpha$.

First, suppose that $S \in I_{\alpha + 1} = (Imp_{\alpha + 1})^{M[G]}$.  Then for some formula $\varphi$ and $a_1, \dots, a_n \in I_\alpha$, we have that $S$ is the unique subset of $I_\alpha$ in $M[G]$ such that $(I_\alpha, \in, S) \models \varphi(a_1, \dots, a_n)$.  But since $S \in M$, we have $S \in N$, and since $N \subset M[G]$, we have that $S$ is also unique in $N$.  Hence $S \in (Imp_{\alpha + 1})^N$.

Conversely, suppose that $S \in  (Imp_{\alpha + 1})^N$.  Then, for some formula $\varphi$ and $a_1, \dots, a_n \in I_\alpha$, we have that $S$ is the unique subset of $I_\alpha$ in $N$ such that $(I_\alpha, \in, S) \models \varphi(a_1, \dots, a_n)$.  Since $N \subset M[G]$, we have $S \in M[G]$, and so by the almost-homogeneity of $\mathbb P$, we have
$$\forces_{\mathbb P} (\exists Z \subseteq I_\alpha) \, ((I_\alpha, \in, Z) \models \varphi(a_1, \dots, a_n)).$$
But then
$$M[G'] \models (\exists Z \subseteq I_\alpha) \, ((I_\alpha, \in, Z) \models \varphi(a_1, \dots, a_n)).$$
Since $S$ was unique in $N$ and $M[G'] \subseteq N$, then the only possible such $Z$ in $M[G']$ is $S$.  Therefore $S \in  (Imp_{\alpha + 1})^{M[G']}=  I_{\alpha + 1} $.
\end{proof}

Now we have that $x \not\in M = (Imp)^N$, but $x \in (HOD)^N$.  This proves the following theorem.

\begin{theorem}
\label{thm-ImpNotHOD}
If $ZF$ is consistent, then so is $ZFC + (Imp \neq HOD)$.
\end{theorem}

\section{Separating $(Imp)^{Imp}$ from $Imp$}

In this section, we produce a model $N$ in which $(Imp)^{Imp} \neq Imp$.

\begin{definition}
The forcing poset $\mathbb{SC}_n$ is the countable-support length $\omega$ iteration of  $(\mathbb S_k)_{k < \omega}$ where (letting $g$ denote the join of the generic reals $g_k$ for $\mathbb S_k$, which we may call the generic real for $\mathbb{SC}_n$),
$$\mathbb S_k = \begin{cases}{\mathbb S} & k \leq n \text{ or } (k=n+2+j \;\&\; j \not\in g); \cr{\mathbb S} \times {\mathbb S} & k = n+1 \text{ or } (k=n+2+j \;\&\; j \in g). \cr \end{cases}$$
(Note that the join of an $\omega$-sequence of reals is defined, as in Definition~\ref{def-infjoin}, in such a way that $g(n)$ depends only on the $\mathbb S_k$-generics for $k < n$.)
\end{definition}

If $M \models V=L$, and $G$ is $\mathbb{SC}_n$-generic over $M$, then by Proposition~\ref{prop-degstruc}, in $M[G]$ the degrees of constructibility form a lattice of height $\omega$ and width 2 (a tower of lines and diamonds) that (uniformly) codes $(n,g)$, where $g$ is the generic real.  In particular, in $M[G]$ there is a unique $\SCn$ generic over $M$ (and no $\mathbb{SC}_m$-generic over $M$ for $m \neq n$).
Furthermore, the lattice of degrees of constructibility in $M[G]$ contains a unique minimal nonzero degree.

\begin{definition}
A self-coding real with base $n$ is any real $x$ such that the degrees of constructibility below $x$ form a lattice coding $x$ in the same way that the generic real $g$ for $\mathbb{SC}_n$ is coded by the degrees of constructibility below $g$.
\end{definition}

It is not hard to see that if $x$ is a self-coding real, the base $n$ is uniquely determined by $x$.
Also, if $x$ is a self-coding real, then $\omega_1^{L[x]} = \omega_1^L$.

\begin{claim}
There is a formula $\varphi_{sc}(n)$ such that
\begin{enumerate}
\item  If $X \subset \omega_1^L$, $(L_{\omega_1^L}, \in, X) \models \varphi_{sc}(n)$, $x = X \cap \omega$, and $\omega_1^{L[x]} = \omega_1^L$, then $x$ is a self-coding real with base $n$;
\item  If $x$ is a self-coding real with base $n$, then there is a unique $X \subset \omega_1^L$ such that $X \cap \omega = x$ and $(L_{\omega_1^L},\in,  X) \models \varphi_{sc}(n)$.
\end{enumerate}
\end{claim}

% Include details?

\begin{proof}
Choose a canonical way of coding the structure of $L_{\omega_1^L}[x]$, for any real $x$, into $Y \subseteq \omega_1^L$, coding the truth predicate of the model $L_{\omega_1^L}[x]$ coded by $Y$ into $Z \subseteq \omega_1^L$, and coding $x$, $Y$ and $Z$ into $X \subseteq \omega_1^L$ in such a way that $X \cap \omega = x$.

Then $\varphi_{sc}(n)$ asserts that $X$ is the canonical code for $L_{\omega_1^L}[X \cap \omega]$, and that the model coded by $X$ satisfies a sentence asserting that the universe is $L_{\omega_1^L}[r]$ for a self-coding real $r$ with base $n$.
\end{proof}

\begin{definition}  If $x$ is a self-coding real with base $n$, then the unique $X \subset \omega_1^L$ such that $X \cap \omega = x$ and $(L_{\omega_1^L}, \in,  X) \models \varphi_{sc}( n)$ is denoted $C(x)$, the canonical code for $L_{\omega_1^L}[x]$.
\end{definition}

This implies the following fact.

\begin{claim}
Let $N$ be a model of $ZF$ in which $x$ is the unique self-coding real with base $n$.  Then in $N$ we have $Imp_{\omega_1^L} = (Imp_{\omega_1^L})^L$ (by Proposition~\ref{prop-HL}),  $C(x) \in Imp_{\omega_1^L + 1}$, and $x \in Imp_{\omega_1^L + 2}$.

\end{claim}

\begin{definition}  The forcing $\mathbb P$ is a countable-support product
$$ \mathbb S \times \prod_{i \in ( \omega_1 \times \omega)} \mathbb P_i$$
where
$$\mathbb P_{(\alpha,n)} = \SCn.$$

If $G$ is $\mathbb P$-generic, then $G$ is equivalent to the sequence of generics
$$\left<G_\mathbb{S}  ,\, G_{i}  \;|\; {i \in (\omega_1 \times \omega)}\right>$$
or to the sequence of generic reals
$$\left<g_\mathbb{S}  ,\, g_{i}  \;|\; {i \in ( \omega_1 \times \omega)}\right>$$

\end{definition}

\begin{claim}
\label{claim-scinN}  If $M \models V=L$, and $G$ is $\mathbb P$-generic over $M$, then in $M[G]$ the only self-coding reals with base $n$ are the generic reals $g_{(\alpha,n)}$.
\end{claim}

\begin{proof}  By Proposition~\ref{prop-degstruc}, the only self-coding real in $M[G_{(\alpha,n)}]$ is the generic real $g_{(\alpha,n)}$.  By the technical lemma (Lemma~\ref{lemma-technical}), any real not in any $M[G_{(\alpha,n)}]$ lies above at least two different minimal (nonzero) $L$-degrees of reals, and therefore is not a self-coding real with base $n$.
\end{proof}

\begin{definition}  $N = M[H]$ is the submodel of $M[G]$ defined by setting $h$ to be the join of $g_{\mathbb S}$ and the reals $g_{0,n}$, and setting
$$H = \left<\,  g_{\mathbb S} , \,  g_{(\alpha,n)} \,|\, {\alpha< \omega_1 \, \& \, (\alpha=0 \; \text{or} \; n \not\in h)} \right>.$$

\end{definition}

\begin{claim}
\label{claim-uniqueinN}
In $N$, there is a unique self-coding real with base $n$ iff $n   \in h$.

\end{claim}

\begin{proof}
By Claim~\ref{claim-scinN}, the only self-coding reals in $M[G]$ are the $g_{(\alpha,n)}$.  By definition of $H$, if $n  \not\in h$, every $g_{(\alpha,n)}$ is in $N = M[H]$, so there are many self-coding reals with base $n$.

If $n   \in h$, by an argument like the proof of Claim~\ref{claim-noextragenerics}, the only $g_{(\alpha,n)}$ in $N$ is $g_{(0,n)}$, so there is a unique self-coding real with base $n$ in $N$.

\end{proof}

\begin{claim}  Let $H' = \left< g_{\mathbb S} , g_{0,n}\right>_{n< \omega}$.  Then:
\begin{enumerate}
\item  $(Imp)^N = M[H']$.  In particular, $g_{\mathbb S} \in (Imp)^N$.
\item  $g_\mathbb S \not\in (Imp)^{M[H']}$.
\end{enumerate}
\end{claim}

\begin{proof}  We know by Proposition~\ref{prop-HL} that in both $N$ and $M[H']$, we have $Imp_{\omega_1^L} = L_{\omega_1^L}$.

For part (1), by Claim~\ref{claim-uniqueinN}, if $n \in h$ then in $N$ there is a unique self-coding real with base $n$, and so there is a unique $X \subseteq \omega_1^L$ such that $(L_{\omega_1^L}, \in, X) \models \varphi_{sc}(n)$, namely $C(g_{0,n})$.  Hence, $C(g_{0,n}) \in Imp_{\omega_1^L + 1}$.  On the other hand, if $n \not\in h$, then we will show that $C(g_{\alpha,n}) \not\in Imp_{\omega_1^L + 1}$, hence there is no $X \in Imp_{\omega_1^L + 1}$ such that $(Imp_{\omega_1^L}, \in, X) \models \varphi_{sc}(n)$

To see this, suppose that $p \in \mathbb P$ and $p$ forces $ (L_{\omega_1^L}, \in, C(g_{\alpha,n}) \models \psi(a_1, \dots, a_k)$.  Since $\mathbb P$ is a product forcing, and $C(g_{\alpha,n})$ is defined from the $\mathbb P_{(\alpha,n)}$ generic, it must be the case that $p(\alpha,n)$, as a condition in $\mathbb P_{(\alpha,n)} = \mathbb S \mathbb C_n$, forces
$(L_{\omega_1^L}, \in, C(g) \models \ \psi(a_1, \dots, a_k)$.  We can extend $p$ to $p'$ such that, for some $(\beta, n) \not\in dom(p)$, we have $p'(\beta, n) = p(\alpha,n)$.
Thus, $p'$ forces $(L_{\omega_1^L}, \in, C(g_{\beta,n}) \models \ \psi(a_1, \dots, a_k)$.  This shows that $C(g_{(\alpha, n)}$ cannot be the unique $X$ such that $(L_{\omega_1^L}, \in, X) \models \ \psi(a_1, \dots, a_k)$, and therefore $C(g_{(\alpha,n)}) \not\in Imp_{\omega_1^L + 1}$.

This shows that  in $N$, the real $h$ is definable over $Imp_{\omega_1^L + 1}$, and therefore $h \in (Imp)^N$.  It follows, since $H'$ is constructible from $h$, that $M[H'] \subseteq (Imp)^N $.

To see the reverse inclusion, by Proposition~\ref{prop-almosthomog} it suffices to note that $N$ is obtained from $M[H']$ as a generic extension for an almost-homogeneous notion of forcing, namely
$$\overline{\mathbb P}  = \left\{ p \hook  \{ ( \alpha, n ) \;|\; 0 < \alpha < \omega_1 \;\&\; n \not \in h\} \;\big|\; p \in \mathbb P\right\}$$
(where $\mathbb P$ is defined in $M$).

\begin{comment1}
To see that $\overline{\mathbb P}$ is almost-homogeneous, let the restrictions $p = \left< p_s, \overline p \right>$ and $q = \left< q_s, \overline q \right>$ be any conditions in $\overline{\mathbb P}$, and let $\gamma < \omega_1$ be such that the supports of $\overline p$ and $\overline q$ are contained in $(\gamma \times \omega)$.  Then the permutation $\varphi$ of $\omega_1 \times \omega$ defined by
$$\varphi(\beta, n) =
\begin{cases}
( \beta, n)   & \text{ if } \beta=0 \text{ or }\beta \geq 2 \gamma;   \cr
( \gamma + m -1, n)    & \text{ if } \beta = m \text{ and } 0 < m < \omega;  \cr
( \gamma + \delta, n)   &   \text{ if } \beta =  \delta \text{ and } \omega \leq \delta < \gamma; \cr
( m + 1, n)   & \text{ if } \beta =  \gamma + m \text{ and } m < \omega;  \cr
( \delta, n)   &   \text{ if } \beta = \gamma + \delta \text{ and } \omega \leq \delta < \gamma; \cr
\end{cases}
$$
induces an automorphism $\overline \varphi$ of $\mathbb P$ that fixes $h$, $H'$, and $\overline{\mathbb P}$.  In $M[H']$, $\overline \varphi$ gives an automorphism of $\overline{\mathbb P}$ that sends $p \hook  \{ ( \alpha, n ) \;|\; 0 < \alpha < \omega_1 \;\&\; n \not \in h\}$ to a condition compatible with $q \hook  \{ ( \alpha, n ) \;|\; 0 < \alpha < \omega_1 \;\&\; n \not \in h\}$.

To see that $N$ is a $\overline{\mathbb P}$-generic extension of $M[H']$, note that $N = M[H'][H'']$ where
$$H'' =  \left< g_{\alpha,n} \,|\, {0 < \alpha< \omega_1 \, \& \,  n \not\in h)} \right>.$$
  Suppose that $D$ is a dense subset of $\overline{\mathbb P}$ in $M[H']$, and $ p = \left<p_s, \underline p \right> \in \mathbb P$.  Choose $q = \left< q_s, \underline q\right> \leq \left< p_s, \underline p \hook \omega \right>$ (so $q$ is a condition for adding $H'$) and $r \in \mathbb P$ such that $q$ forces
$$\overline r \in D \text{ and } \overline r \leq p \hook  \{ ( \alpha, n ) \;|\; 0 < \alpha < \omega_1 \;\&\; n \not \in h\}  \text{ where } \overline r = r \hook  \{ ( \alpha, n ) \;|\; 0 < \alpha < \omega_1 \;\&\; n \not \in h\}.$$
Note that if $ \alpha > 0$ and $r(\alpha, n) \not\leq p(\alpha, n)$, then $q \forces n \in h$.
Define $p' \leq p$ by $p' = \left< q_s, \underline p'\right>$ where
$$\underline p'( \alpha, n) =
\begin{cases}
q(\alpha, n)   & \text{ if } \alpha = 0 ;\cr
r(\alpha, n)   & \text{ if }  \alpha > 0 \text{ and } r(\alpha, n) \leq p(\alpha, n)  ;\cr
p(\alpha, n)   & \text{ otherwise.}    \cr
\end{cases}
  $$
Then $p'$ forces the generic filter adding $H''$ to meet $D$.
\end{comment1}

For part (2), $H'$ is generic over $M$ for the product forcing $\mathbb S \times \displaystyle\prod_{n \in\omega} \mathbb S \mathbb C_n$.  Since this is product forcing, $M[H']$ is a generic extension of $M[\left<g_{0,n} \, | \, {n \in\omega}\right>]$ by the almost-homogeneous forcing $(\mathbb S)^M$, which adds the generic real $g_\mathbb S$.  Therefore in $M[H']$, by Proposition~\ref{prop-almosthomog}, we have that $g_{\mathbb S} \not\in Imp$.

\end{proof}

This shows that in $N$ we have $g_{\mathbb S} \in Imp$ and $g_{\mathbb S} \not\in (Imp)^{Imp}$, proving the following theorem.

\begin{theorem}
\label{thm-ImpImpNotImp}
If $ZF$ is consistent, so is $ZFC + ((Imp)^{Imp} \neq Imp)$.
\end{theorem}

\section{Proof of Technical Lemma}
\label{sec-technical}

In this section we prove the technical lemma.

\begin{rlemma}
Suppose $M \models V=L$, the poset $\mathbb P \in M$ is as described in Section~\ref{sec-prelim}, and $G$ is $\mathbb P$-generic over $M$, where $G = \left< G_i \;|\; i \in I\right>$.  If $x$ is a real in $ M[G]$ and for all $i \in I$ we have $x \not\in M[G_i]$, then the $L$-degree of $x$ lies above at least two minimal (nonzero) $L$-degrees of reals.
\end{rlemma}

To establish notation and intuition, we begin by reviewing Sacks forcing.

Sacks forcing conditions are perfect trees, binary trees in which branching nodes are dense.  A subtree is a stronger condition.  The generic $G$ is equivalent to the generic real $g$, the unique real that is a branch through all the trees in $G$.

\begin{definition}  If $T \subseteq 2^{<\omega}$ is downward closed, and $\sigma \in T$, we say $\sigma$ \emph{splits} in $T$ if $\sigma\tie 0 \in T$ and $\sigma \tie 1 \in T$.  We may call $\sigma$ a \emph{splitting node} of $T$.

A \emph{perfect tree} is a downward closed $T \subseteq 2^{<\omega}$ such that for every $\sigma \in T$ there is some $\tau \supseteq \sigma$ that splits in $T$.

A branch of $T$ is $b \in 2^\omega$ such that, for all $n < \omega$, the restriction $b \hook n$ is an element of $T$.  The set of all branches of $T$ is denoted $[T]$.

Sacks forcing $\mathbb S$ has as conditions perfect trees, ordered by $T' \leq T$ ($T'$ is stronger than $T$) iff $T' \subseteq T$.

If $G \subseteq \mathbb S$ is $\mathbb S$-generic over $M$, then the $\mathbb S$-generic real $g$ is defined in $M[G]$ by $g = \bigcap \{ [T] \;|\;T \in G\}$.

\end{definition}

If $G \subseteq \mathbb S$ is $\mathbb S$-generic over $M$, then $G = \{ T \in (\mathbb S)^M \;|\; g \in [T]\}$.  Hence, $M[G] = M[g]$.

\begin{definition}  Let $T$ be a perfect tree.  The root, or stem, of $T$ is the shortest $\tau \in T$ that splits in $T$.

For $\sigma \in 2^n$, we define $rt_\sigma (T)$ by induction on $n$:

If $\sigma = \left< \right>$ is the empty sequence, then $rt_\sigma(T)$ is the root of $T$.

Inductively, $rt_{\sigma \tie i}(T)$ is the shortest $\tau \supseteq (rt_\sigma(T)) \tie i$ that splits in $T$.

\end{definition}

\begin{remark}  The collection of splitting nodes $\{rt_\sigma(T)\;|\; \sigma \in 2^{<\omega}\}$ comprises an isomorphic copy of the complete binary tree $2^{<\omega}$ inside $T$.

Any branch $b$ through $T$ is determined by $\{\sigma \;|\; rt_\sigma(T) \subset b\}$, and any branch $b'$ through $2^{\omega}$ determines a branch through $T$, given by the downward closure of $\{ rt_\sigma(T) \;|\; \sigma \subset b'\}$.

\end{remark}

We use the $rt_\sigma(T)$ to construct fusion sequences (defined below), an essential tool for Sacks forcing arguments.

\begin{definition}  For $T \in \mathbb S$ and $n < \omega$, the $n^{th}$ splitting level of $T$ is $S_n(T) = \{ rt_\sigma(T) \;|\; \sigma \in 2^n\}$.

If $T' \leq T$, and $S_m(T') = S_m(T)$ for all $m < n$, we say $T' \leq_n T$.

A \emph{fusion sequence} for $\mathbb S$ is a decreasing (with respect to the partial ordering) sequence of conditions $\left< T_m \;|\; m \in \omega\right>$ such that
$$(\forall n)(\exists k_n)(\forall m, m') \; (k_n \leq m \leq m' \;\Longrightarrow\; T_{m'} \leq_n T_m).$$
The \emph{fusion} of the sequence is $\bigcap \{ T_m \;|\; m < \omega\}$.

\end{definition}

\begin{remark}
The set $S_n(T)$ is a maximal antichain of nodes of $T$; any branch through $T$ extends exactly one element of $S_n(T)$.

$T' \leq_0 T$ iff $T' \leq T$.

The fusion $T$ of a fusion sequence $\left< T_m \;|\; m \in \omega\right>$ is a condition, $T = \displaystyle{\bigwedge_{m < \omega} T_m}$.
Furthermore, for all $m$ we have $T \leq T_m$, and for all $m \geq k_n$ we have $T \leq_n T_m \leq_n T_{k_n}$.

\end{remark}

Although $\mathbb S$ is not countably closed, we can use closure under fusions of fusion sequences in place of countable closure to prove, for example, that $\mathbb S$ does not collapse $\omega_1$, and that any function $f:\omega \to \omega$ in a generic extension by $\mathbb S$ is dominated by a function in the ground model.

Fusion sequences are also used to prove the property for which Sacks forcing was developed:  If $G$ is Sacks generic over $M$, then every element of $M[G]$ is either in $M$ or equivalent (equidefinable using parameters from $M$) to $G$.  In particular, if $M \models V=L$, then the generic real $g$ is of minimal (nonzero) degree of constructibility.

Typically, constructing a fusion sequence uses the notion of restriction:

\begin{definition}

If $\tau \in T$, then $T_\tau$, sometimes called the restriction of $T$ to $\tau$, is
$$\{ \rho \in T \;|\; \rho \subseteq \tau \;\text{or}\; \tau \subseteq \rho\}.$$

If $\sigma \in 2^{<\omega}$, then $T_{(\sigma)} = T_{rt_\sigma(T)}$.

\end{definition}

\begin{remark}  If $\sigma \in 2^n$, then $rt_\sigma(T) \in S_n(T)$ is the root of $T_{(\sigma)}$.

For each $n < \omega$, the collection $\{ [T_{(\sigma)}] \;|\; \sigma \in 2^n \}$ is a partition of the branches of $T$.  The collection $\{T_{(\sigma)} \;|\; \sigma \in 2^n\}$ is a maximal antichain of conditions extending $T$.

The condition $T$ forces that  exactly one element of $S_n(T) = \{ rt_\sigma(T) \;|\; \sigma \in 2^n \}$ is an initial segment of $g$, and $g$ is a branch through exactly one $T_{(\sigma)}$ for $\sigma \in 2^n$.

For all $T$, $T'$ in $\mathbb S$,  $$T' \leq_n T \;\Longleftrightarrow\; (\forall \sigma \in 2^n)\,( T'_{(\sigma)} \leq T_{(\sigma)}).$$

\end{remark}

It will be convenient to use this alternative characterization of $\leq_n$ when we generalize the definition to products and iterations.

\begin{remark}

  Suppose $D \subseteq \mathbb S$ is an open dense set.  Given $T$ and $n$, we can extend every $T_{(\sigma)}$ for $\sigma \in 2^n$ to a condition $S(\sigma) \in D$, and put those extensions together to form $S \leq_n T$ such that $S_{(\sigma)} = S(\sigma) \in D$ for each $\sigma \in 2^n$.  (Formally, $S = \displaystyle\bigcup_{\sigma \in 2^n} S(\sigma)$.)  Thus, $S$ forces the generic to contain one of finitely many elements $S_{(\sigma)}$ of $D$.

\begin{comment2}

If $D_n$ is an open dense set for each $n \in \omega$, we can begin with any condition $T$ and build a fusion sequence, $\left< T_n \;|\; n < \omega \right>$ with $T_0 =T$, $T_{n+1} \leq_n T_n$, and, for all $\sigma \in 2^n$, $(T_{n+1})_{(\sigma)} \in D_n$.  The fusion $S$ of this fusion sequence forces that, for each $n$, the generic contains one of finitely many elements $(T_{n+1})_{(\sigma)}$ of $D_n$.

If $x$ is a term for a function from $\omega$ to $M$, and $D_n$ is the set of conditions forcing a value for $x(n)$, then $S$ forces each $x(n)$ to lie within a finite set of possible values.  Hence, if $x : \omega \to \omega_1$, the range of $x$ is contained in some countable set in $M$; this shows $\mathbb S$ does not collapse $\omega_1$.  If $x : \omega \to \omega$, then $S$ determines a ground model function $f : \omega \to \omega$ such that $S$ forces $x$ to be dominated by $f$; $f(n)$ is an upper bound for the possible values of $x(n)$.

Suppose that $x$ is a term for an element of $2^\omega$ that is not in the ground model $M$.  To show $x$ is equivalent to the generic real $g$, construct a fusion sequence below any condition $T$, so that the fusion $S$ provides a method for using $x$ to determine $\{ \sigma \in 2^{<\omega} \;|\; S_{(\sigma)} \in G\}$ (and hence, to determine $g$).  To do this, when extending $R = T_n$ to $T_{n+1} \leq_n R$, instead of extending each individual $R_{(\sigma)}$ to lie in some dense set, extend each pair $R_{(\sigma)}$ and $R_{(\tau)}$ (for $\sigma \neq \tau$) to force contradictory facts about $x$ (that is, for some $k$, one forces $x(k)=0$ and the other forces $x(k) = 1$).  Then the resulting $\overline R$ will force that, for $\sigma \in 2^n$, we have $\overline R_{(\sigma)} \in G$ iff $\overline R_{(\sigma)}$ forces only correct facts about $x$; since $S \leq_n T_{n+1} \leq_n \overline R$, we have $S_{(\sigma)} \in G$ iff $S_{(\sigma)}$ forces only correct facts about $x$.  In this way $x$ recovers $\{ \sigma \;|\; rt_\sigma(S) \subset g\}$, and therefore $x$ recovers $g$.

We can always $\leq_n$-extend $R$ in this way:  Since $x$ is forced not to be in $M$, there must be some $k$ such that $R_{(\sigma)}$ does not decide the value of $x(k)$.  Then we can extend $R_{(\tau)}$ to decide the value of $x(k)$, and $R_{(\sigma)}$ to decide the opposite value.  Repeating this for all pairs $\sigma \neq \tau $ from $2^n$, we produce the desired $T_{n+1}$.

\end{comment2}

\end{remark}

To extend the fusion technique to products and iterations of $\mathbb S$, we use coordinatewise definitions of restrictions, fusion sequences, and fusions.

In particular, suppose $p = \left< p(i) \;|\; i \in I \right>$ is a condition.  To define an analogue of $p_{(\sigma)}$, we break up $\sigma$ into finitely many subsequences $\sigma_k$, for each $k$ choose a coordinate $i$, and replace $p(i)$ with the restriction $(p(i))_{(\sigma_k)}$.

To organize this, we introduce notation for finite and infinite joins.

\begin{definition}
\label{def-join}
For $x,y \in 2^\omega$, the \emph{join} of $x$ and $y$ is $x \oplus y$, defined by
$$(x\oplus y)(n) = \begin{cases} x(k) & \text{ if } n = 2k;\cr y(k) & \text{ if } n = 2k+1.\cr\end{cases}$$

We make a similar definition for $\sigma \in 2^m$ and $\tau \in 2^m$ or $\tau \in 2^{m-1}$:  $\sigma \oplus \tau$ has domain $2m$ in the first case, and $2m-1$ in the second, and
$$(\sigma \oplus \tau) (n) = \begin{cases} \sigma(k) & \text{ if } n = 2k;\cr \tau(k) & \text{ if } n = 2k+1.\cr\end{cases}$$

If $z \in 2^{\leq \omega}$, we can view $z$ as a join, and define its left and right parts:  If $z = x \oplus y$, then $\ell(z)=x$ and $r(z)=y$.

\end{definition}

\begin{definition}
\label{def-infjoin}

Let $[\, , \,] : \omega \times \omega \to \omega$ be a computable bijection, increasing in each coordinate, such that $[0,0] = 0$, $[0,1] = 1$, and otherwise $[m,n] > max(m,n)$.

If $\sigma$ is a sequence of length at most $\omega$, for each $n < \omega$ define the sequence $c(\sigma, n)$ by $c(\sigma, n) (m) = \sigma([n,m])$.  If $[n,m] $ is not in the domain of $\sigma$, then  $m$ is not in the domain of $c(\sigma, n)$.

The least $n$ such that, for $\sigma$ of length $k$ and all $m \geq n$, the domain of $c(\sigma, m)$ is empty, is denoted $W(k)$.

If, for each $n$,  $x_n$ is a sequence of length $\omega$, the join $\displaystyle\bigoplus_{n < \omega} x_n$ is the sequence $x$ defined by $x([n,m]) = x_n(m)$.

\end{definition}

The bijection $[ \, , \,]$ allows us to view a (possibly partial) function $\sigma$ on $\omega$ as a function on $\omega \times \omega$.  If we view $\omega \times \omega$ as a two-dimensional grid, then $c(\sigma, n)$ is the restriction of $\sigma$ to the $n^{th}$ column of the grid.

If $\sigma$ is a finite sequence, then $c(\sigma, n)$ is a finite sequence for all $n$, and for all but finitely many $n$ we have $c(\sigma, n) = \left< \right>$.  If we view $\sigma$ of length $k$ as a partial function on the $\omega \times \omega$ grid, then $W(k)$ is the width of the domain of $\sigma$, that is, the number of columns having nonempty intersection with the domain of $\sigma$.

Taking the join of $\left< x_n \;|\; n < \omega\right>$ is the reverse process, viewing each $x_n$ as a function on the $n^{th}$ column of the grid.

\begin{remark}
If $\sigma = \displaystyle\bigoplus_{n < \omega} x_n$, then $c(\sigma, n) = x_n$.

\end{remark}

For combinations of iterations and products of Sacks forcing over a model of $V=L$, we want to employ the method of fusion sequences to analyze the degrees of constructibility in the generic extension.

We define fusion sequences and fusions coordinatewise, with an inductive component to the definition in the case of iteration.

\begin{definition}

\begin{enumerate}

\item  A decreasing sequence $\left< p_m \;|\; m < \omega \right>$ in $\mathbb S \times \mathbb S$ is a fusion sequence if it is a fusion sequence in each coordinate; that is, if $p_m = (S_m, T_m)$, then both $\left< S_m \;|\; m < \omega \right>$ and $\left< T_m \;|\; m < \omega \right>$ are fusion sequences in $\mathbb S$.

Its fusion is the coordinatewise fusion, $\displaystyle{\bigwedge_{m < \omega}p_m = \left( \bigcap_{m < \omega}S_m, \, \bigcap_{m< \omega}T_m\right)}$.

\item  A fusion sequence for an iteration $\mathbb Q$ of length $\alpha$, and its fusion, are defined coordinatewise.  Inductively:

 A decreasing sequence $\left< p_m \;|\; m < \omega \right>$ is a fusion sequence if for all $\beta < \alpha$, the sequence  $\left< p_m\hook \beta \;|\; m < \omega \right>$ is a fusion sequence in $\left< \mathbb Q_\gamma \;|\; \gamma < \beta \right>$, and its fusion forces  $\left< p_m(\beta) \;|\; m < \omega \right>$ to be a fusion sequence in $\mathbb Q_\beta$.

The fusion of the sequence is the condition $p = \displaystyle{\bigwedge_{m < \omega} p_m}$ such that, for all $\beta < \alpha$, we have that $p \hook \beta = \displaystyle{\left(\bigwedge_{m < \omega} p_m\hook \beta\right)}$ and  $p(\beta)$ denotes the fusion $\displaystyle{\bigwedge_{m  < \omega} \left(p_m(\beta)\right)} $.

\item  A fusion sequence for $\mathbb P$, and its fusion, are defined coordinatewise:

A decreasing sequence $\left< p_m \;|\; m < \omega \right>$ is a fusion sequence if for all $i$ the sequence $\left< p_m(i) \;|\; m < \omega \right>$ is a fusion sequence for $\mathbb P_i$.

Its fusion is defined by
$\displaystyle{\left( \bigwedge_{m < \omega} p_m\right)(i) =  \bigwedge_{m < \omega} \left(p_m(i)\right)         }$.

\end{enumerate}

\end{definition}

As with $\mathbb S$, the fusion, or infimum, of a fusion sequence is a condition.

For constructing fusion sequences in this setting, we want to generalize the definitions of $T_{(\sigma)}$ and $\leq_n$.

For countable products and iterations, we can decompose $\sigma$ into subsequences $c(\sigma, n)$, and use a fixed enumeration $\{i(n) \,|\, n <\omega\}$ of the support of the product or iteration to make a coordinatewise definition of $p_{(\sigma)}$.

For an uncountable product or iteration, we define a notion $p_{(\sigma, \vec s)}$, where $\vec s$ identifies the coordinates to which we associate those subsequences $c(\sigma, n)$ that are nonempty.

\begin{definition}

\begin{enumerate}

\item  Suppose $p = (T_0, T_1) \in \mathbb S \times \mathbb S$.

For $\sigma \in 2^{<\omega}$, we define
$p_{(\sigma)} = ( (T_0)_{(\ell(\sigma))}, (T_1)_{(r(\sigma))})$.

We define $p \leq_n q$ iff $(\forall \sigma \in 2^n)\,(p_{(\sigma)} \leq q_{(\sigma)})$.  (As in Defintion~\ref{def-join}, $\ell(q_\sigma)$ and $r(\sigma)$ denote the left and right parts of $\sigma$.)

\item  For $\mathbb Q$
 of countable length $\alpha$, fix an enumeration $\{ \beta_m \;|\; m < \omega\}$ of $\alpha$, such that $\beta_0 = 0$.

 For $p \in \mathbb Q$ and $\sigma \in 2^{<\omega}$ we define $p_{(\sigma)}$ coordinatewise:  $p_{(\sigma)}(\beta_m)$ is a term for $(p(\beta_m))_{c(\sigma,m)}$.

 We define $p \leq_n q$ iff $(\forall \sigma \in 2^n)\,(p_{(\sigma)} \leq q_{(\sigma)})$.

\item  For $p \in \mathbb P$, $\sigma \in 2^n$, and $\vec s = \left< i(k) \;|\; k < \delta\right>$, where $W(n) \leq \delta \leq \omega$,
we define $p_{(\sigma, \vec s)}$ by $p_{(\sigma, \vec s)}(i(k)) = p(i(k))_{( c(\sigma, k )}$, and for $i \not\in \{i(k) \;|\; k < \delta\}$, we set $p_{(\sigma, \vec s)}(i) = p(i)$.

We define $p \leq_{n,\vec s} q$ iff $(\forall \sigma \in 2^n)\,(p_{(\sigma, \vec s)} \leq q_{(\sigma, \vec s)})$.

\end{enumerate}

\end{definition}

\begin{remark}  For $\mathbb S \times \mathbb S$ and $\mathbb Q$, if $p \in G$, then the sequence $\left< \sigma \;|\; p_{(\sigma)} \in G\right>$ is equivalent to the generic $G$.

For $\mathbb P$, if $\vec s = \left< i(k) \;|\; k < \omega \right>$ and $p \in G$, the sequence $\left< \sigma \;|\; p_{(\sigma, \vec s)} \in G\right>$ is equivalent to the portion of the generic $\left< G_{i(k)} \;|\; k < \omega \right>$.

This is also true coordinatewise:  In $\mathbb S \times \mathbb S$, from $\{ \ell(\sigma) \;|\; p_{(\sigma)} \in G\}$ and $p$, we can recover $g_0$, and similarly for $g_1$.  In $\mathbb Q$, from $\{ c(\sigma, k) \;|\; p_{(\sigma)} \in G\}$, $G \hook \beta_k$, and $p$, we can recover $g_{\beta_k}$.  In $\mathbb P$, from $\{ c(\sigma, k) \;|\; p_{(\sigma)} \in G\}$ and $p$, we can recover $g_{i(k)}$.

For $\mathbb S \times \mathbb S$ and $\mathbb Q$, if $\vec p = \left< p_m \;|\; m < \omega\right>$ is a sequence of conditions such that
$$(\forall n)(\exists k_n)(\forall m, m') \; (k_n \leq m \leq m' \;\Longrightarrow\; p_m' \leq_n p_m),$$
then $\vec p$ is a fusion sequence.  Furthermore, if $p$ is its fusion, then for all $m \geq k_n$ we have $p \leq_n p_m$.

For $\mathbb P$, if $\vec p = \left< p_m \;|\; m < \omega\right>$ is a sequence of conditions and $\vec s = \left< i(k) \;|\; k <\omega\right>$ is an enumeration of $\displaystyle\bigcup \{ supp(p_m) \;|\; m < \omega\}$ such that
$$(\forall n)(\exists k_n)(\forall m, m') \; (k_n \leq m \leq m' \;\Longrightarrow\; p_m' \leq_{(n, \vec s)} p_m),$$
then $\vec p$ is a fusion sequence.  Furthermore, if $p$ is its fusion, then for all $m \geq k_n$ we have $p \leq_{(n, \vec s)} p_m$.

Note, in this case, that $p \leq_{(n,\vec s)}q $ is equivalent to $p \leq_{(n,\vec s \hook m)}q $, for any $m\geq W(n)$.  We will use this in constructing fusion sequences, when we may need to find $p \leq_{(n,\vec s)}q $ although only some initial segment of $\vec s$ has been defined.

To produce $r \leq_n p$ such that the restrictions $r_{(\sigma)}$ (or $r_{(\sigma , \vec s)}$) for $\sigma \in 2^n$ all have some given property, we wish, as in the case of Sacks forcing, to extend each $p_{(\sigma)}$ individually, and then put the results together to form $r$.  However, we can no longer extend the $p_{(\sigma)}$ independently; extending $p_{(\sigma)}$ generally changes $p_{(\tau)}$ for $\tau \neq \sigma$.  To facilitate extending the $p_{(\sigma)}$ sequentially, for $q \leq p_{(\sigma)}$ we define the amalgamation of $q$ into $p$ above $\sigma$, essentially the result of extending $p_{(\sigma)}$ and then plugging the extension $q$ back into $p$.

The amalgamation of $q$ into $p$ above $\sigma$ will be the maximal (weakest) $r \leq_n p$ such that $r_{(\sigma)} = q$.

\end{remark}

\begin{definition}

\begin{enumerate}

\item[0.]  For $\sigma \in 2^n$, $T \in \mathbb S$, and $S \leq T_{(\sigma)}$, the amalgamation of $S$ into $T$ above $\sigma$ is $Am_\sigma (T,S) = \displaystyle{ S \cup \bigcup\{T_{(\tau)} \;|\; \tau \in 2^n \; \& \; \tau \neq \sigma \} }$.

\item For $\sigma \in 2^n$, and $\mathbb S \times \mathbb S$  conditions  $p = (T_0, T_1)$ and $q = (S_0, S_1) \leq p_{(\sigma)}$, we define $Am_\sigma(p,q)$, the amalgamation of $q$ into $p$ above level $n$, to be the coordinatewise amalgamation
$ ( Am_{\ell(\sigma)}( T_0 , S_0 ) ,\,  Am_{r(\sigma)}(T_1, S_1 )) $.

\item
For $\sigma \in 2^n$, $p \in \mathbb Q$, and $q \leq p{(\sigma)}$, we define $Am_\sigma(p,q)$, the amalgamation of $q$ into $p$ above $\sigma$, inductively:  Letting $r$ denote $Am_\sigma(p,q)$,
$$r(\beta_m) = \begin{cases}   Am_{c(\sigma,m)} (p(\beta_m), q(\beta_m))  &  \text{ if }(\forall \beta_k < \beta_m) \, ( p(\beta_k)_{(c(\sigma, k) )} \in G_{\beta_k} ) ) ; \cr
 p(\beta_m) &  \text{ otherwise. } \cr \end{cases}$$

\item
For $p \in \mathbb P$, $\sigma \in 2^n$, $\vec s = \left< i(k) \;|\; k < \delta\right>$ (with $\delta \geq W(n)$), and $q \leq p_{(\sigma, \vec s)}$, the amalgamation of $q$ into $p$ above $(\sigma, \vec s)$ is defined to be the coordinatewise amalgamation:  Letting $r$ denote the amalgamation
$Am_{(\sigma, \vec s)}(p,q)$, we define $r(i(k)) = Am_{( c(\sigma, k) )} (p(i(k)), q(i(k)))$, and for $i \not\in \{i(k) \;|\; k < \delta\}$, we set $r(i) = q(i)$.

\end{enumerate}
\end{definition}

\begin{remark}
\label{remark-res-amal}  In the case of $\mathbb S$, for $\sigma \in 2^n$, and $R = Am_\sigma(T,S)$, we have $R_{(\sigma)} = S$, and for $\tau \in 2^n$ with $\tau \neq \sigma$, we have $R_{(\tau)}= T_{(\tau)}$.

In the case of $\mathbb S \times \mathbb S$, for $\sigma \in 2^n$, $q  \leq p_{(\sigma)}$, and $r = Am_\sigma(p,q)$, we have $r_{(\sigma)} = q$.  For $\tau \in 2^n$ with $\tau \neq \sigma$, we have $r_{(\tau)} \leq p_{(\tau)}$, but we do not in general have $r_{(\tau)} = p_{(\tau)}$.  (However, if $\ell(\tau) \neq \ell(\sigma)$, we have equality in the first coordinate, and if $r(\tau) \neq r(\sigma)$, we have equality in the second coordinate.)

In the case of $\mathbb Q$ as in clause (2) above, for $\sigma \in 2^n$, $q  \leq p_{(\sigma)}$, and $r = Am_\sigma(p,q)$, we have $r_{(\sigma)} = q$.  For $\tau \in 2^n$ with $\tau \neq \sigma$, we have $r_{(\tau)} \leq p_{(\tau)}$, but we do not in general have $r_{(\tau)} = p_{(\tau)}$.  However, we do have the following (which will be used later):  If $c(\sigma, 0) \neq c(\tau, 0)$, then $r_{(\tau)} = p_{(\tau)}$.

\begin{comment3}
  An illustrative case is the two-step iteration of Sacks forcing $\mathbb Q = \left< \mathbb Q_0, \mathbb Q_1\right>$, where $\mathbb Q_0$ and $\mathbb Q_1$ are both Sacks forcing.  A condition in $\mathbb Q$ is a pair $ p =\left< T, T' \right>$, where $T$ is a perfect tree in $M$, and $T'$ is a term for a perfect tree in $M[G_0]$. Suppose that $\sigma  = \left< 0, 0  \right> = \left< 0 \right> \oplus \left< 0 \right>$, and $q = \left< S, S'\right> \leq p_{(\sigma)} = \left< T_{(\left< 0 \right>)}, T'_{(\left< 0 \right>)} \right>$.  (For purposes of illustration we are using the pairwise join, rather than the infinite join, to decompose $\sigma$.)  If $r$ is the amalgamation of $q$ into $p$ above $\sigma$, then

$r_{( \left< 0\right> \oplus \left< 0 \right> )} = r_{(\sigma)} = \left<  S, S' \right> = q$,

$r_{( \left< 0\right> \oplus \left< 1 \right> )} =  \left<  S, T'_{(\left< 1\right>)}  \right> $,

$r_{( \left< 1\right> \oplus \left< 0 \right> )} =  \left< T_{(\left< 1\right>)}  , T'_{(\left< 0\right>)}  \right>  = p_{( \left< 1\right> \oplus \left< 0 \right>)}$, and

$r_{( \left< 1\right> \oplus \left< 1 \right> )} =  \left<T_{(\left< 1\right>)}  , T'_{(\left< 1\right>)}  \right>  = p_{( \left< 1\right> \oplus \left< 1 \right>)}$.

\noindent This is precisely what is needed for $r \leq_2 p$ with $r_{(\sigma)} = q$ to be maximal (as weak as possible).  Here $\ell(\tau)$ is playing the role of $c(\tau,0)$, and where $\ell(\tau) = \left< 1 \right> \neq \ell(\sigma)$, we have $r_{(\tau)} = p_{(\tau)}$, as claimed.

\end{comment3}

In the case of $\mathbb P$ as in clause (3) above, for $q  \leq p_{(\sigma, \vec s)}$, and $r = Am_{(\sigma, \vec s)}(p,q)$, we have $r_{(\sigma,\vec s)} = q$.  For $\tau \in 2^n$ with $\tau \neq \sigma$, we have $r_{(\tau,\vec s)} \leq p_{(\tau ,\vec s)}$, but we do not in general have $r_{(\tau, \vec s)} = p_{(\tau, \vec s)}$.  However, derived from the above, we do have the following (which will be used later):

Suppose $i = i(k)$ for some $k < \delta$, and for $\tau \in 2^n$ we set $b(\tau) = c(c(\tau, k),0)$.  (That is, $b(\tau)$ is the subsequence of $\tau$ that determines the initial path of the generic real for $(\mathbb P_i)_0$.)  If $b(\tau) \neq b(\sigma)$ then $r_{(\tau,\vec s)}(i) = p_{(\tau,\vec s)}(i)$.

\end{remark}

\begin{lemma}
\label{lemma-amal}
  Let $M \models V=L$.

  In $M$, let $\mathbb P = \displaystyle\prod_{i \in I} \mathbb P_i$ be a countable-support product where each $\mathbb P_i$ is a countable-support iteration $\left< \mathbb Q_\beta \;|\; \beta <\alpha_i \right>$ of countable length $\alpha_i$, such that each $\mathbb Q_\beta$ is either $\mathbb S$ or $\mathbb S \times \mathbb S$ (which one may depend on the generic sequence below $\beta$).  Let $x$ be a $\mathbb P$-term for a function from $\omega$ to the ordinals.

Let $i \in I$, and $p \in \mathbb P$, such that $p \forces x \not\in M[\left< G_j \;|\; j \neq i \right>]$.  Let $n < \omega$, and $\vec s = \left< i(k) \;|\; k < \delta\right>$ be a finite sequence from $I$ of length at least $W(n)$, with $i(0) = i$.

For $\tau \in 2^n$, set $b(\tau) = c(c(\tau, 0),0)$.
Suppose $\sigma, \tau \in 2^n$ and $b(\tau) \neq b(\sigma)$.  Then there are a condition $r \leq_{n, \vec s} p$, a number $d < \omega$, and ordinals $\gamma \neq \gamma'$ such that $r_{(\sigma, \vec s)} \forces x(d) = \gamma$ and $r_{(\tau, \vec s)} \forces x(d) = \gamma'$.

\end{lemma}

\begin{proof}  Suppose by way of contradiction that, for all $d < \omega$, and all $q \leq p_{\sigma, \vec s}$ such that $q \forces x(d) = \gamma$, if $r = Am_{\sigma, \vec s}(p,q)$, we have $r_{(\tau, \vec s)} \forces x(d) = \gamma$.  By Remark~\ref{remark-res-amal}, since $r_{(\tau,\vec s)}(i) = p_{(\tau,\vec s)}(i)$ (and since amalgamation is defined coordinatewise), the condition $r_{(\tau, \vec s)}$ is determined by $p$ and $q \hook \{j \;|\; j \neq i\}$.  Specifically, $r_{(\tau, \vec s)} = \left( Am_{(\sigma, \vec s)} (p, \overline q)  \right)_{(\tau, \vec s)}$, where $ \overline q(i) = p_{(\tau,\vec s)}(i)$ and $ \overline q(j) = q(j)$ for $j \neq i$.

Then $p_{\sigma, \vec s}$ forces that $x(d) = \gamma$ if and only if there is $ q \in G\hook\{ j \;|\; j \neq i\}$ such that $\left(Am_{\sigma, \vec s}(p, \overline q) \right)_{(\tau, \vec s)} \forces x(d) = \gamma$, where $ \overline q(i) = p_{(\tau,\vec s)}(i)$ and $ \overline q(j) = q(j)$ for $j \neq i$.  But this contradicts $p \forces x \not\in M[\left< G_j \;|\; j \neq i \right>]$.

 Hence, we can find $d < \omega$, $\gamma$, and $q \leq p_{(\sigma, \vec s)}$, such that $q \forces x(d) = \gamma$, but $\left(Am_{\sigma, \vec s}(q,p)\right)_{(\tau, \vec s)} \not\forces x(d) = \gamma$.  Choose $q' \leq \left(Am_{\sigma, \vec s}(q,p)\right)_{(\tau, \vec s)}$ such that $q' \forces x(d) = \gamma'$ for $\gamma ' \neq \gamma$, and let $r = Am_{\tau, \vec s}(Am_{\sigma, \vec s}(p,q), q')$.  Then $r_{(\tau, s)} = q' \forces x(d) = \gamma'$, and $r_{(\sigma ,\vec s)} \leq \left(Am_{\sigma, \vec s}(p,q)\right)_{\sigma, \vec s} = q \forces x(d) = \gamma$, as desired.

\end{proof}

\begin{proposition}
\label{prop-aboveminimal}
Let $M \models V=L$.

  In $M$, let $\mathbb P = \displaystyle\prod_{i \in I} \mathbb P_i$ be a countable-support product where each $\mathbb P_i$ is a countable-support iteration $\left< \mathbb Q_\beta \;|\; \beta <\alpha_i \right>$ of countable length $\alpha_i$, such that each $\mathbb Q_\beta$ is either $\mathbb S$ or $\mathbb S \times \mathbb S$ (which one may depend on the generic sequence below $\beta$).  Let $x$ be a term for a function from $\omega$ to the ordinals.

Let $i \in I$, and $p \in \mathbb P$, such that $p \forces x \not\in M[\left< G_j \;|\; j \neq i \right>]$.  Then $p \forces (G_i)_0 \leq_L x$.

\end{proposition}

\begin{proof}  We define a fusion sequence $\left< p_m \;|\; m < \omega\right>$ and an enumeration $\vec s = \left< i(k) \;|\; k < \omega \right>$ of $\displaystyle\bigcup_{m < \omega} supp(p_m)$, such that from the fusion $\displaystyle\bigwedge p_m$, the sequence $\vec s$, and $x$, we can recover the generic real $(g_i)_0$ for $(\mathbb Q_i)_0$.

At step $m$ we define $p_m$ and $i(m)$, using a diagonalization strategy to insure the range of $\vec s$ is $\displaystyle\bigcup_{m < \omega} supp(p_m)$.

We will guarantee we have a fusion sequence by making $p_{m+1} \leq_{m+1, \vec s \hook (m+1)} p_m$.

Let $p_0 = p$, and $i(0) = i$.

Inductively, suppose $p_m$ and $\vec s_m = \left< i(0), \dots, i(m)\right>$ have been defined, with the following property:  Say that conditions $q$ and $r$ are separated by $x$ if there are a number $d < \omega$ and ordinals $\gamma \neq \gamma'$ such that $ p \forces x(d) = \gamma$ and $q \forces x(d) = \gamma'$.  For $\sigma \in 2^m$, as in Lemma~\ref{lemma-amal}, define $b(\sigma) = c(c(\sigma, 0),0)$.  Then, for all $\sigma, \tau \in 2^m$, if $b(\sigma) \neq b(\tau)$, then $(p_m)_{(\sigma, \vec s_m)}$ and $(p_m)_{(\tau, \vec s_m)}$ are separated by $x$.

Choose $i(m+1)$ according to our diagonalization strategy.

By Lemma~\ref{lemma-amal}, for any $\sigma, \tau \in 2^{m+1}$ with $b(\sigma) \neq b(\tau)$, there is $r \leq_{m+1, \vec s \hook (m+1)} p_m$ such that $r_{(\sigma, {\vec s_{m+1}})}$ and $r_{(\tau, {\vec s_{m+1}})}$ are separated by $x$.   Therefore, by a finite iteration of choosing $\leq_{m+1}$ extensions, we may choose $p_{m+1} \leq_{m+1, \vec s \hook (m+1)} p_m$ such that, for all $\sigma, \tau \in 2^{m+1}$ with $b(\sigma) \neq b(\tau)$, we have that
$(p_{m+1})_{(\sigma, {\vec s_{m+1}})}$ and $(p_{m+1})_{(\tau, {\vec s_{m+1}})}$ are separated by $x$.

Let $q$ be the fusion $\displaystyle\bigwedge p_m$.  Then, $q$ forces $x$ to recover the generic real $(g_i)_0$ for $(\mathbb Q_i)_0$ as follows:  Let $\rho \in 2^n$, and $m$ such that if $\tau \in 2^m$ then $b(\tau) \in 2^n$.  Then $p_m$ forces that $rt_{\rho}((p_m(i))(0)) \subset (g_i)_0$
iff there is $\tau \in 2^m$ with $b(\tau) = \rho$ such that $(p_m)_{( \tau, \vec s \hook m)}$ forces no incorrect facts about $x$; that is, whenever $(p_m)_{( \tau, \vec s \hook m)} \forces x(d) = \gamma$, then in fact $x(d) = \gamma$.

\end{proof}

\begin{proposition}
\label{prop-technical}

  In $M$, let $\mathbb P = \displaystyle\prod_{i \in I} \mathbb P_i$ be a countable-support product where each $\mathbb P_i$ is a countable-support iteration $\left< \mathbb Q_\beta \;|\; \beta <\alpha_i \right>$ of countable length $\alpha_i$, such that each $\mathbb Q_\beta$ is either $\mathbb S$ or $\mathbb S \times \mathbb S$ (which one may depend on the generic sequence below $\beta$).  Let $x$ be a term for a function from $\omega$ to ordinals.

  Then, in $M[G]$, one of:
  \begin{enumerate}
  \item  $x \in M$;
  \item  $x \in M[G_i]$, where $G_i$ is $\mathbb P_i$-generic for some $i \in I$; or
  \item  $x$ lies above at least two minimal (nonzero) $L$-degrees of reals.
  \end{enumerate}
  This proves the technical lemma.

\begin{remark}  Unlike the previous proposition, this depends on the fact that the domain of $x$ is countable.
\end{remark}

\end{proposition}

\begin{proof}  Let $p$ be any condition, and extend $p$ so one of
\begin{enumerate}
\item  $p \forces (\forall i \in I) \,(x \in M[\left< G_j \;|\; j \neq i \right>])$;
\item  For some $k \in I$, we have $p \forces  \,(x \not\in M[\left< G_j \;|\; j \neq k \right>])$, and  $p \forces (\forall i \in I) \,(i \neq k \;\Longrightarrow\; x \in M[\left< G_j \;|\; j \neq i \right>])$;
\item  For some $i_0, i_1 \in I$, we have $p \forces  \,(x \not\in M[\left< G_j \;|\; j \neq i_0 \right>])$ and $p \forces  \,(x \not\in M[\left< G_j \;|\; j \neq i_1 \right>])$.
\end{enumerate}

In case (3), by Proposition~\ref{prop-aboveminimal}, $p$ forces $x$ to lie above the (nonzero) minimal $L$-degrees added by $G_{i_0}$ and $G_{i_1}$.  We show that in case (1), we can extend $p$ to force $x \in M$, and in case(2), to force $x \in M[G_{i_0}]$.

In each case, we build a decreasing sequence of conditions $\left< p_n \;|\; n < \omega \right>$, and simultaneously build an enumeration $\vec s = \left< i(n) \;|\; n < \omega \right>$ of $\displaystyle\bigcup_{n < \omega} supp(p_n)$.  For $q \in \mathbb P$, we let $q[-n]$ denote $q \hook \{i \;|\; i \not\in \{i(0), i(1), \dots i(n)\}\}$.  We let $\mathbb P[-n]$ denote $\displaystyle\prod_{i \not\in \{ i(0), \dots, i(n)\}} \mathbb P_i$.

\bigskip

For case (1), choose $i(0)$.  Since $p \forces (x \in M[\left< G_j \;|\; j \neq i_0 \right>])$, we can choose $p_0 \leq p$ and a term $x_0$ for forcing with  $\mathbb P[-0]$ such that $p_0\forces x=x_0$ and $p_0 \forces x_0(0) = \gamma_0$.

Since $x_0$ is a term for forcing with $ \mathbb P[-0]$ and $p_0 \forces (x_0 \in M[\left< G_j \;|\; j \neq i(1) \right>])$, we also have $p_0 \forces (x_0 \in M[\left< G_j \;|\; j \not\in \{i(0), i(1)\} \right>])$, and $p_0[-0]  \forces (x_0 \in M[\left< G_j \;|\; j \not\in \{i(0), i(1)\} \right>])$.
Therefore we can choose $p_1[-0] \leq p_0[-0]$ and a term $x_1$ for forcing with $\mathbb P[-1]$ such that $p_1[-0] \forces x_0 = x_1$ and $p_1[-0] \forces x_1(1) = \gamma_1$.  Expand $p_1[-0] $ to $p_1 \leq p_0$ by setting $p_1(i_0)  = p_0(i_0)$.

Then $p_1 \leq p_0 \leq p$, $p_1(i(0)) = p_0(i(0))$, and $p_1 \forces (x = x_1 \;\&\; x(0)= \gamma_0 \;\&\; x(1) = \gamma_1)$.

Inductively, assume that we have $p \geq p_0 \geq \cdots \geq p_{n-1}$, $x_0, x_1, \dots, x_{n-1}$, and $\gamma_0, \gamma_1, \dots, \gamma_{n-1}$ such that for all $m' < m < n$:
\begin{enumerate}
\item $p_m(i(m')) = p_{m'}(i(m'))$;
\item $x_m$ is a term for forcing with $\mathbb P[-m]$ and $p_m \forces x = x_m$;
\item $p_m \forces x(m) = \gamma_m$.
\end{enumerate}
Then, as before, we can choose $p_n[-(n-1)] \leq p_{n-1}[-(n-1)]$ and a term $x_n$ for forcing with $\mathbb P[-n]$ such that $p_n[-(n-1)] \forces x_{n-1} = x_n$, and $p_n[-(n-1)] \forces x_n(n) =\gamma_n$.  Expand $p_n[-(n-1)]$ to $p_n$ by, for $m < n$, setting $p_n(i(m)) = p_m(i(m))$.  This preserves the inductive hypothesis.

Finally, let $q$ be the limit of the $p_n$:  $q(i(n)) = p_n(i(n))$.  Then, for all $n$, we have $q \forces x(n) = \gamma_n$.  Therefore, $q \forces x \in M$.

\bigskip

For case (2), we combine the construction of case (1) with the construction of a fusion sequence for $\mathbb P_{k}$.  During the course of the construction, we construct $\vec s = \left< i(n) \;|\; n < \omega \right>$ enumerating $\left(\displaystyle\bigcup_{n < \omega} supp(p_n)\right) - \{k\}$.  For $q \in \displaystyle\prod_{i \neq k} \mathbb P_i$ and $r \in \mathbb P_{k},$ we let $q \tie r$ denote the condition defined by setting $(q \tie r)(i) = q(i)$ for $i \neq k$ and $(q \tie r)(k) = r$.

Choose $p_0 \leq p$ such that $p_0 \forces x = x_{\left< \right>}$, where $x_{\left< \right>}$ is a term for forcing with $\mathbb P[-0]$, and $p_0\forces x_{\left< \right>}(0) =   \gamma_{\left< \right>}$.

Inductively, assume we have constructed a decreasing sequence of conditions $\left< p_m \;|\; m < n\right>$, a collection of terms $\left< x_\sigma \;|\; \sigma \in 2^m, m < n\right>$, and a collection of ordinals $\left< \gamma_\sigma \;|\; \sigma \in 2^m, m < n\right>$ such that, for all $m' < m < n$,
\begin{enumerate}
\item $p_m(k) \leq_{m' } p_{m'}(k)$;
\item $p_m(i(m')) = p_{m'}(i(m'))$;
\item For $\sigma \in 2^m$, $x_\sigma$ is a term for forcing with $\mathbb P[-m]$;
\item For $\sigma \in 2^m$, $((p_m(k))_{(\sigma)}) \tie (p_m \hook I - \{k, i(0), \dots, i(m-1) \} ) \forces x_\sigma (m) = \gamma_\sigma$;
\item For $\tau \in 2^{m-1}$, $((p_m(k))_{(\tau\tie \ell)} )\tie (p_m \hook I - \{k, i(0), \dots, i(m-1) \}) \forces x_\tau = x_{\tau \tie \ell}$.
\end{enumerate}

We can extend this to $n$ as follows:

Let $q_0$ denote $p_{n-1}(k)$ and $r_0$ denote $p_{n-1} \hook (I - \{ k, i(0), i(1), \dots i(n-1)\})$, and enumerate $2^n $ as $ \{ \sigma_j \;|\; j = 1, \dots, b\}$.

Inductively, for $j= 1, \dots, b$, if $\sigma_j = \tau \tie \ell$, choose a condition $(r  \leq  (q_{j-1})_{(\sigma_j)} )\tie r_{j-1}$, a term $x_{\sigma_j}$ for forcing with $\mathbb P[-n]$, and an ordinal $\gamma_{\sigma_j }$ such that $r \forces x_\tau = x_{\sigma_j}$ and $r \forces x_{\sigma_j}(n) = \gamma_{\sigma_j}$.  Let $q_j = Am_{\sigma_j} (q_{j-1}, r(k))$ and $r_j = r \hook (I - \{ k, i(0), i(1), \dots i(n-1)\})$.

Define $p_n$ by $p_n(k) = q_k$, $p_n(i(j)) = p_{n-1}( i(j))$ for $j=0, \dots, n-1$, and $p_n \hook (I - \{ k, i(0), \dots, i(n-1)\} ) = r_b $.

Now, define $q \leq p$ by setting $q(k) = \bigwedge_{n \in \omega} p_n(k)$ and $q(i(n)) = p_n(i(n))$.  Then $q$ forces that, for all $n$, we have $(x(n) = \gamma_\sigma) \;\Longleftrightarrow\; ((q(k))_{(\sigma)} \in G_k)$; that is, $q$ forces $x \in M[G_k]$.

\end{proof}

The technical lemma proved here
(Lemma~\ref{lemma-technical})
is a special case of the general analysis of degrees in generic extensions by forcing notions built from Sacks forcing.  A reasonably general result, proved using the ideas in case 2 of the proof of  Proposition~\ref{prop-technical}, is stated in the following proposition.

This proposition concerns generalized iterations as defined in  Groszek and Jech \cite{GJ}.  For a well-founded partial ordering $I$, a generalized iteration $\left< \mathbb P_i \;|\; i \in I \right>$ is defined in the natural way, so that $\mathbb P_i$ is a term for a partial ordering in the generic extension by $\left< \mathbb P_j \;|\; j < i\right>$.  Generalized iterations encompass products, iterations, and various combinations of products and iterations.

\begin{proposition}
Let $M \models V=L$.  In $M$, let $I$ be any well-founded, $\omega_2$-like partial ordering, and let $\mathbb P$ be a countable-support generalized iteration along $I$ such that every $\mathbb P_i$ is forced to be either $\mathbb S$ or the trivial forcing (which one may depend on the generic below $i$).  Suppose $G$ is $\mathbb P$-generic over $M$.

If $x$ is any real in $M[G]$, then $x \equiv_L  \left< X_n \;|\; n < \omega\right>$ where, for some $\{ i(n) \;|\; n < \omega\} \subseteq I$ in $M$, and some $\{ F(n) \;|\; n < \omega\}$ in $M$ such that each $F_n$ is a finite set of finite binary sequences,
$$X_n = \begin{cases}
G_{i(n)}   &   \text{ if } (\exists \sigma \in F(n)) \, (\sigma \subset \bigoplus_{m < n} X_m)  \cr
\emptyset   &   \text{ otherwise.}  \cr
\end{cases}$$
Furthermore, for all $i \in I$,
$$( G_i \leq_L x) \;\Longleftrightarrow\; ( (\exists n) \, (G_i \leq_L X_n)) \;\Longleftrightarrow\;  ( (\exists n) \, (   X_n = G_{i(n)} \;\&\; i \leq i(n)     )   ).$$

\end{proposition}

\end{document}